\documentclass[12pt]{amsart}
\usepackage{amssymb}
\usepackage{enumitem}
\usepackage{graphicx}
\usepackage[all]{xy}
\def\pmod #1{\ ({\rm{mod}}\ #1)}

\def\bg{\bigg}
\def\({\bg(}
\def\){\bg)}

\def\sgn{{\rm sgn}}

\def\sgn{{\rm sgn}}

\theoremstyle{plain}
\newtheorem{theorem}{Theorem}

\newtheorem{lemma}{Lemma}
\newtheorem{corollary}{Corollary}
\newtheorem{conjecture}{Conjecture}
\theoremstyle{definition}

\theoremstyle{remark}
\newtheorem{remark}{Remark}

\makeatletter
\@namedef{subjclassname@2010}{%
  \textup{2010} Mathematics Subject Classification}
\makeatother
 \vspace{4mm}

\begin{document}
 \baselineskip=17pt
\hbox{} {}
\medskip
\title[A conjecture of Zhi-Wei Sun on determinants over finite fields]
{A conjecture of Zhi-Wei Sun on determinants over finite fields}
\date{}
\author[H.-L. Wu, Y.-F. She and H.-X. Ni]{Hai-Liang Wu,Yue-Feng She and He-Xia Ni*}

\thanks{2020 {\it Mathematics Subject Classification}.
Primary 11C20; Secondary 11L05, 11R29.
\newline\indent {\it Keywords}. determinants, the Legendre symbol, finite fields.
\newline \indent  The first author was supported by the National Natural Science Foundation of China (Grant No. 12101321 and Grant No. 11971222) and the Natural Science Foundation of the Higher Education Institutions of Jiangsu Province (Grant No. 21KJB110002). The third author was supported by the National Natural Science Foundation of China (Grant No. 12001279).}
\thanks{*Corresponding author.}

\address {(Hai-Liang Wu) School of Science, Nanjing University of Posts and Telecommunications, Nanjing 210023, People's Republic of China}
\email{\tt whl.math@smail.nju.edu.cn}

\address {(Yue-Feng She) Department of Mathematics, Nanjing
University, Nanjing 210093, People's Republic of China}
\email{{\tt she.math@smail.nju.edu.cn}}

\address{(He-Xia Ni)  Department of Applied Mathematics, Nanjing Audit
University, Nanjing 211815, People's Republic of China}
\email{\tt nihexia@yeah.net}

\begin{abstract}
In this paper, we study certain determinants over finite fields. Let $\mathbb{F}_q$ be the finite field of $q$ elements and let $a_1,a_2,\cdots,a_{q-1}$ be all nonzero elements of $\mathbb{F}_q$. Let $T_q=\left[\frac{1}{a_i^2-a_ia_j+a_j^2}\right]_{1\le i,j\le q-1}$ be a matrix over $\mathbb{F}_q$. We obtain the explicit value of $\det T_q$. Also, as a consequence of our result, we confirm a conjecture posed by Zhi-Wei Sun.
\end{abstract}
\maketitle
\section{Introduction}
Let $R$ be a commutative ring. Then for any $n\times n$ matrix $M=[a_{ij}]_{1\le i,j\le n}$ with $a_{ij}\in R$, we use $\det M$ or $|M|$ to denote the determinant of $M$.

Let $p$ be an odd prime and let $(\frac{\cdot}{p})$ be the Legendre symbol. Carlitz \cite{carlitz} studied the following matrix
$$C_p(\lambda)=\bigg[\lambda+\left(\frac{i-j}{p}\right)\bigg]_{1\le i,j\le p-1}\ \ \ \ \ (\lambda\in\mathbb{C}).$$
Carlitz \cite[Theorem 4]{carlitz} proved that the characteristic polynomial of $C_p(\lambda)$ is
$$P_{\chi}(t)=(t^2-(-1)^{(p-1)/2}p)^{(p-3)/2}(t^2-(p-1)\lambda-(-1)^{(p-1)/2}).$$
Later Chapman \cite{chapman,evil} investigated many interesting variants of $C_p$. Moreover, Chapman \cite{evil} posed a challenging conjecture on the determinant of the $\frac{p+1}{2}\times\frac{p+1}{2}$ matrix
$$E_p=\bigg[\(\frac{j-i}{p}\)\bigg]_{1\le i,j\le\frac{p+1}{2}}.$$
Due to the difficulty of the evaluation on $\det E_p$, Chapman called this determinant ``evil" determinant. Finally, by using sophisticated matrix decompositions, Vsemirnov \cite{M1,M2} solved this problem completely.

Along this line, in 2019 Sun \cite{ffadeterminant} studied the following matrix
$$S_p=\bigg[\left(\frac{i^2+j^2}{p}\right)\bigg]_{1\le i,j\le \frac{p-1}{2}},$$
and Sun \cite[Theorem 1.2(iii)]{ffadeterminant} showed that $-\det S_p$ is always a quadratic residue modulo $p$. In the same paper, Sun also investigated the matrix
$$A_p=\bigg[\frac{1}{i^2+j^2}\bigg]_{1\le i,j\le \frac{p-1}{2}}.$$
Sun \cite[Theorem 1.4(ii)]{ffadeterminant} proved that when $p\equiv3\pmod4$ the $p$-adic integer $2\det A_p$ is always a quadratic residue modulo $p$. In addition, let
$$T_p=\bigg[\frac{1}{i^2-ij+j^2}\bigg]_{1\le i,j\le p-1}.$$
Sun \cite[Remark 1.3]{ffadeterminant} posed the following conjecture.
\begin{conjecture}[Zhi-Wei Sun]\label{Conjecture of Sun}
	Let $p\equiv2\pmod3$ be an odd prime.
	Then $2\det T_p$ is a quadratic residue modulo $p$.
\end{conjecture}
Let $\mathbb{F}_q$ be the finite field of $q$ elements and let $$\mathbb{F}_q^{\times}=\mathbb{F}_q\setminus\{0\}=\{a_1,a_2,\cdots,a_{q-1}\}.$$
Motivated by this conjecture, we define a matrix $T_q$ over $\mathbb{F}_q$ by
$$T_q=\bigg[\frac{1}{a_i^2-a_ia_j+a_j^2}\bigg]_{1\le i,j\le q-1}.$$

We obtain the following generalized result.
\begin{theorem}\label{Thm. A}
Let $q\equiv 2\pmod 3$ be an odd prime power and let
$$T_q=\bigg[\frac{1}{a_i^2-a_ia_j+a_j^2}\bigg]_{1\le i,j\le q-1}.$$
Then
$$\det T_q=(-1)^{\frac{q+1}{2}}2^{\frac{q-2}{3}}\in\mathbb{F}_p,$$
where $p$ is the characteristic of $\mathbb{F}_q$.
\end{theorem}
\begin{remark}
We give two examples here. Note that we also view $T_p$ as a matrix over $\mathbb{F}_p$ if $p$ is an odd prime.

{\rm (i)} If $p=5$, then
$$T_p=\frac{11}{596232}=\frac{1}{2}=-2.$$

{\rm (ii)} If $p=11$, then
\begin{align*}
T_p=\frac{393106620416000000}{23008992710579652367225919172202284572822491031943}=\frac{4}{6}=2^3.
\end{align*}
\end{remark}

As a direct consequence of our theorem, we confirm Sun's conjecture.
\begin{corollary}\label{Corollary A}
	Conjecture \ref{Conjecture of Sun} holds.
\end{corollary}
The outline of this paper is as follows. In section 2, we will prove some lemmas which are the key elements in the proof of our theorem. The proofs of Theorem \ref{Thm. A} and Corollary \ref{Corollary A} will be given in section 3.
\maketitle
\section{Some Preparations}

Given any polynomials $A(T),B(T)\in\mathbb{F}_q[T]$, we say that $A(T)$ and $B(T)$ are equivalent (denoted by $A(T)\sim B(T)$) if $A(x)=B(x)$ for each $x\in\mathbb{F}_q$.

Let $\chi_3(\cdot)=(\frac{\cdot}{3})$ be the quadratic character modulo $3$. We first have the following lemma.
\begin{lemma}\label{Lemma equivalent reduced polynomials}
Let $q\equiv 2\pmod 3$ be an odd prime power and let
\begin{equation}\label{Eq. definition of G(T)}
G(T)=1+\frac{1}{3}\sum_{k=2}^{q-2}\left(\chi_3(k)+\chi_3(1-k)\right)T^{k-1}
+\frac{1}{3}T^{q-2}-\frac{2}{3}T^{q-1}.
\end{equation}
Then
$$(T^2+T+1)^{q-2}\sim G(T).$$
\end{lemma}
\begin{proof}
We first show that $T^2+T+1$ is irreducible over $\mathbb{F}_q[T]$. Set $q=p^r$ with $p$ prime and $r\in\mathbb{Z}^{+}$. As $q\equiv2\pmod3$, clearly $p\equiv2\pmod3$ and $2\nmid r$. Hence
$$(-3)^{\frac{q-1}{2}}=(-3)^{\frac{p-1}{2}(1+p+p^2+\cdots+p^{r-1})}=\left(\frac{-3}{p}\right)^{1+p+p^2+\cdots+p^{r-1}}=(-1)^r=-1.$$
This implies that $-3$ is not a square over $\mathbb{F}_q$. Suppose now $T^2+T+1$ is reducible over $\mathbb{F}_q[T]$. Then there exists an element $\alpha\in\mathbb{F}_q$ such that $\alpha^2+\alpha+1=0$. This implies $(2\alpha+1)^2=-3$, which is a contradiction. Hence $T^2+T+1$ is irreducible over $\mathbb{F}_q[T]$. Moreover, since
$$T^q-T=\prod_{x\in\mathbb{F}_q}\left(T-x\right),$$
we have $T^2+T+1\nmid T^q-T$ and hence $T^2+T+1$ is coprime with $T^q-T$.

Now via a computation, we obtain
\begin{equation*}
(T^2+T+1)^2G(T)\equiv T^2+T+1\equiv (T^2+T+1)^q\pmod{(T^q-T)\mathbb{F}_q[T]}.
\end{equation*}
As $T^2+T+1$ is coprime with $T^q-T$, we obtain
$$(T^2+T+1)^{q-2}\equiv G(T)\pmod{(T^q-T)\mathbb{F}_q[T]}.$$
This implies
$$(T^2+T+1)^{q-2}\sim G(T).$$
In view of the above, we have completed the proof.
\end{proof}
We need the following lemma (cf. \cite[Lemma 10]{K2}).
\begin{lemma}\label{Lemma formula for determinants}
Let $R$ be a commutative ring and let $n$ be a positive integer. Set $P(T)=p_{n-1}T^{n-1}+\cdots+p_1T+p_0\in R[T]$. Then
$$
\det[P(X_iY_j)]_{1\le i,j\le n}
=\prod_{i=0}^{n-1}p_i\prod_{1\le i<j\le n}\left(X_j-X_i\right)\left(Y_j-Y_i\right).
$$
\end{lemma}
Now let $m$ be a positive integer. We introduce some basic facts on the permutations over $\mathbb{Z}/m\mathbb{Z}$. Fix an integer $a$ with $(a,m)=1$. Then the map $x\ {\rm mod}\ m \mapsto ax\ {\rm mod}\ m$ induces a permutation $\pi_a(m)$ over $\mathbb{Z}/m\mathbb{Z}$. Lerch \cite{ML} determined the sign of this permutation.
\begin{lemma}\label{Lemma permutation}
Let $\sgn(\pi_a(m))$ denote the sign of the permutation $\pi_a(m)$. Then  $$\sgn(\pi_a(m))=\begin{cases}(\frac{a}{m})&\mbox{if $m$ is odd},\\1&\mbox{if}\
m\equiv2\pmod4,\\(-1)^{\frac{a-1}{2}}&\mbox{if}\ m\equiv0\pmod4,\end{cases}$$
where $(\frac{\cdot}{m})$ denotes the Jacobi symbol if $m$ is odd.
\end{lemma}
Recall that
$$\mathbb{F}_q^{\times}=\mathbb{F}_q\setminus\{0\}=\left\{a_1,a_2,\cdots,a_{q-1}\right\}.$$
The map $a_j\mapsto a_j^{-1}$ $(j=1,2,\cdots,q-1)$ induces a permutation $\sigma_{-1}$ on $\mathbb{F}_q^{\times}$. We also need the following lemma.
\begin{lemma}\label{Lemma Inv of the inverse permutation}
Let notations be as above. Then
$$\sgn(\sigma_{-1})=\sgn(\pi_{-1}(q-1))=(-1)^{\frac{q+1}{2}}.$$
\end{lemma}

\begin{proof}
	Fix a generator $g$ of $\mathbb{F}_q^{\times}$. Let $f$ be the bijection on $\mathbb{F}_q^{\times}$ which sends $a_j$ to $g^j$ $(j=1,2,\cdots,q-1)$. Then it is easy to see that
	$$\sgn(\sigma_{-1})=\sgn(f\circ \sigma_{-1}\circ f^{-1}).$$
Note that $f\circ \sigma_{-1}\circ f^{-1}$ is the permutation on $\mathbb{F}_q^{\times}$ which sends $g^j$ to $g^{-j}$ $(j=1,2,\cdots,q-1)$. This permutation indeed corresponds to the permutation $\pi_{-1}(q-1)$ over $\mathbb{Z}/(q-1)\mathbb{Z}$ which sends $j$ {\rm mod} $(q-1)$ to $-j$ {\rm mod } $(q-1)$. Now our desired result follows from Lemma \ref{Lemma permutation}.

This completes the proof.
\end{proof}

\section{Proof of The Main Result}
{\bf Proof of Theorem \ref{Thm. A}.} Recall that
$$T_q=\bigg[\frac{1}{a_i^2-a_ia_j+a_j^2}\bigg]_{1\le i,j\le q-1}.$$
By Lemma \ref{Lemma permutation}
$$
\det T_q
=(-1)^{\frac{q-1}{2}}\det\bigg[\frac{1}{a_i^2+a_ia_j+a_j^2}\bigg]_{1\le i,j\le q-1}.
$$
Also,
$$
\det\bigg[\frac{1}{a_i^2+a_ia_j+a_j^2}\bigg]_{1\le i,j\le q-1}
=\prod_{j=1}^{q-1}\frac{1}{a_j^2}\cdot\det\bigg[\frac{1}{(a_i/a_j)^2+a_i/a_j+1}\bigg]_{1\le i,j\le q-1}.
$$
Since $q\equiv2\pmod3$, we have $a_i^2+a_ia_j+a_j^2\neq0$ for any $1\le i,j\le q-1$.
Hence for any $1\le i,j\le q-1$ we have
$$
\frac{1}{(a_i/a_j)^2+a_i/a_j+1}=
\left((a_i/a_j)^2+a_i/a_j+1\right)^{q-2}.
$$
By Lemma \ref{Lemma equivalent reduced polynomials} we have $(T^2+T+1)^{q-2}\sim G(T)$, where $G(T)$ is defined by (\ref{Eq. definition of G(T)}). Hence
$$\left((a_i/a_j)^2+a_i/a_j+1\right)^{q-2}=G(a_i/a_j),$$
for any $1\le i,j\le q-1$. As $(a_i/a_j)^{q-1}=1$ for any $1\le i,j\le q-1$ , we have
$$G(a_i/a_j)=H(a_i/a_j),$$
where
$$
H(T)=G(T)-\frac{2}{3}+\frac{2}{3}T^{q-1}=\frac{1}{3}+\frac{1}{3}
\sum_{k=2}^{q-2}\left(\chi_3(k)+\chi_3(1-k)\right)T^{k-1}
+\frac{1}{3}T^{q-2}.
$$

Let
$$S(T)=\prod_{1\le j\le q-1}\left(T-a_j\right)$$
and let $S'(T)$ be the formal derivative of $S(T)$. It is easy to verify that
$$S(T)=\prod_{1\le j\le q-1}\left(T-a_j\right)=T^{q-1}-1.$$
By this it is clear that $S'(T)=(q-1)T^{q-2}=-T^{q-2}$ and
\begin{equation}\label{Eq. Production of all aj}
\prod_{1\le j\le q-1}a_j=-1.	
	\end{equation}
By the above we obtain
\begin{equation}\label{Eq. A in the proof of theorem}
\det T_q=(-1)^{\frac{q-1}{2}}\det\left[H(a_i/a_j)\right]_{1\le i,j\le q-1}.
\end{equation}
By Lemma \ref{Lemma formula for determinants} we know that $\det[H(a_i/a_j)]_{1\le i,j\le q-1}$ is equal to
$$
\frac{1}{3^{q-1}}\prod_{k=2}^{q-2}\left(\chi_3(k)+\chi_3(1-k)\right)\prod_{1\le i<j\le q-1}\left(a_j-a_i\right)\left(\frac{1}{a_j}-\frac{1}{a_i}\right).
$$
We first consider the product
$$
\prod_{1\le i<j\le q-1}\left(a_j-a_i\right)\left(\frac{1}{a_j}-\frac{1}{a_i}\right).
$$
By Lemma \ref{Lemma Inv of the inverse permutation} it is easy to see that
\begin{equation*}
\prod_{1\le i<j\le q-1}\left(a_j-a_i\right)\left(\frac{1}{a_j}-\frac{1}{a_i}\right)=
(-1)^{\frac{q+1}{2}}\prod_{1\le i<j\le q-1}\left(a_j-a_i\right)^2.
\end{equation*}
It is easy to verify that
\begin{align*}
	\prod_{1\le i<j\le q-1}(a_j-a_i)^2
	&=(-1)^{\frac{(q-1)}{2}}\prod_{1\le i\neq j \le q-1}(a_j-a_i)\\
	&=(-1)^{\frac{(q-1)}{2}}\prod_{1\le j\le q-1}\prod_{i\neq j}(a_j-a_i)\\
	&=(-1)^{\frac{(q-1)}{2}}\prod_{1\le j\le q-1}S'(a_j)\\
	&=(-1)^{\frac{q-1}{2}}\prod_{1\le j\le q-1}\frac{-1}{a_j}=(-1)^{\frac{q+1}{2}}
\end{align*}
The last equality follows from (\ref{Eq. Production of all aj}). Hence
\begin{equation}
	\prod_{1\le i<j\le q-1}\left(a_j-a_i\right)\left(\frac{1}{a_j}-\frac{1}{a_i}\right)=1.
\end{equation}
We now turn to the product
$$
\prod_{k=2}^{q-2}\left(\chi_3(k)+\chi_3(1-k)\right).
$$
By definition
$$\chi_3(k)+\chi_3(1-k)=\begin{cases}1&\mbox{if}\ k\equiv 0,1\pmod 3,
\\-2&\mbox{if}\ k\equiv 2\pmod 3.
\end{cases}$$
Hence
\begin{equation}\label{Eq. C in the proof of theorem}
\prod_{k=2}^{q-2}\left(\chi_3(k)+\chi_3(1-k)\right)=(-2)^{\frac{q-2}{3}}.
\end{equation}
In view of (\ref{Eq. A in the proof of theorem})-(\ref{Eq. C in the proof of theorem}), we obtain
$$\det T_q=(-1)^{\frac{q+1}{2}}2^{\frac{q-2}{3}}\in\mathbb{F}_p,$$
where $p$ is the characteristic of $\mathbb{F}_q$. This completes the proof.\qed

{\bf Proof of Corollary \ref{Corollary A}.} Let $p\equiv2\pmod3$ be an odd prime. Then by Theorem \ref{Thm. A} we have
$$\left(\frac{\det T_p}{p}\right)
=\left(\frac{-1}{p}\right)^{\frac{p+1}{2}}\left(\frac{2}{p}\right)^{\frac{p-2}{3}}
=\left(\frac{2}{p}\right).$$

This completes the proof. \qed

\end{document}